\documentclass[12pt]{article}
\usepackage{fullpage,amssymb,amsfonts,amsmath,amstext,amsthm,amscd,graphicx,verbatim,enumerate}

\begin{document}

\newtheorem{theorem}{Theorem}[section]
\newtheorem{definition}[theorem]{Definition}
\newtheorem{lemma}[theorem]{Lemma}
\newtheorem{proposition}[theorem]{Proposition}
\newtheorem{corollary}[theorem]{Corollary}

\theoremstyle{definition}
\newtheorem{remark}[theorem]{Remark}
\newtheorem*{remarks}{Remarks}

\numberwithin{equation}{section}

\newcommand{\Z}{\mathbb{Z}}
\newcommand{\R}{\mathbb{R}}
\newcommand{\C}{\mathbb{C}}
\newcommand{\N}{\mathbb{N}}
\newcommand{\essinf}{\operatorname*{ess\: inf}}
\newcommand{\inte}{\operatorname{int}}

\newcommand{\Real}{\mathop{\rm Re}\nolimits}
\newcommand{\Imag}{\mathop{\rm Im}\nolimits}

\newcommand{\eqn}{\begin{equation}}
\newcommand{\eqnend}{\end{equation}}

\title{Chaotic dynamics of a quasiregular sine mapping}
\author{Alastair N. Fletcher and Daniel A. Nicks}

\maketitle

\begin{abstract}
This article studies the iterative behaviour of a quasiregular mapping $S:\R^d\to\R^d$ that is an analogue of a sine function. We prove that the periodic points of $S$ form a dense subset of $\R^d$.  We also show that the Julia set of this map is $\R^d$ in the sense that the forward orbit under $S$ of any non-empty open set is the whole space $\R^d$. The map $S$ was constructed by Bergweiler and Eremenko~\cite{BE} who proved that the escaping set $\{x: S^k(x)\to\infty \mbox{ as } k\to\infty\}$ is also dense in $\R^d$.
\end{abstract}

\section{Introduction}

In complex dynamics, it is now well known that there are transcendental entire functions such as $e^z$ and $\pi\sin z$ for which the Julia set is the whole complex plane (see \cite{Mi} and \cite{Do}, the former confirming a conjecture of Fatou from 1926). This means that, in some sense, these functions are chaotic everywhere.
The recent interest in the iteration of quasiregular mappings on $\R^d$ has led to a search for examples which exhibit interesting dynamics. We will recall the definition of quasiregularity in Section~\ref{sect:qr}, while we refer the reader to \cite{BeCMFT}, and also to \cite{Be,BN}, for an introduction to the iteration theory of quasiregular mappings.

In this article, we show that a quasiregular version of a sine function constructed by Bergweiler and Eremenko in \cite{BE} has iterates that behave chaotically everywhere. In particular, we show that the whole space $\R^d$ possesses three properties typically associated with the Julia set of a holomorphic function.

\begin{theorem}
\label{mainthm}
Let $d\geq 2$ be an integer. Then there exists a quasiregular mapping ${S:\R^d \to \R^d}$ such that
\begin{enumerate}[(i)]
\item The periodic points of $S$ are dense in $\R^d$.
\item $S$ has the blowing-up property everywhere in $\R^d$; that is,
\[ \bigcup_{k=0}^\infty S^k(U) = \R^d, \quad \mbox{for any non-empty open set } U \subseteq \R^d. \]
Equivalently, for any $x\in\R^d$, the backward orbit $O^-(x)=\bigcup_{k=0}^\infty S^{-k}(x)$ is dense in $\R^d$.
\end{enumerate}
\end{theorem}

We remark that all the periodic points of $S$ are repelling because the mapping $S$ is constructed  to be locally expanding everywhere in $\R^d$; see Section~\ref{sect:qrsine}.

The escaping set $I(f) = \{ x: f^k(x)\to\infty \mbox{ as } k\to\infty \}$ of a quasiregular map $f$ was studied in \cite{BFLM} and \cite{FN}. For holomorphic functions $f$, the boundary of the escaping set $\partial I(f)$ coincides with the Julia set \cite{E}.

\begin{corollary}\label{cor}
For the mapping $S:\R^d \to \R^d$ in Theorem~\ref{mainthm}, we have $\partial I(S)=\R^d$.
\end{corollary}

The mapping $S$ appearing in Theorem~\ref{mainthm} was constructed by Bergweiler and Eremenko in \cite{BE} (see Section~\ref{sect:qrsine} below). By considering the dynamics of such a map they proved the following counterintuitive result, in which a \emph{hair} is a simple curve tending to infinity in $\R^d$ with one endpoint in $\R^d$: There exists a representation of $\R^d$ as a union of hairs such that two hairs can intersect only at a common endpoint, while the union of hairs without their endpoints has Hausdorff dimension one. Furthermore, these hairs minus their endpoints lie in the escaping set $I(S)$ \cite[Lemma~3]{BE}. It follows that $I(S)$ is dense in $\R^d$.

In view of this result, it is now easy to establish Corollary~\ref{cor} (but see also Remark~\ref{correm} below). The periodic points of $S$ clearly belong to the complement $I(S)^c$. Hence Theorem~\ref{mainthm}(i) and the above imply that both the sets $I(S)$ and $I(S)^c$ are  dense in $\R^d$ and thus $\partial I(S)=\R^d$.

\begin{remarks}\mbox{}
\begin{enumerate}
\item The quasiregular mapping $S$ is said to be of \emph{transcendental type} because $S(x)$ does not tend to a limit as $x\to\infty$. The Julia set of a quasiregular mapping of transcendental type was recently defined in \cite{BN} as the set of all points such that the complement of the forward orbit of any neighbourhood has conformal capacity zero. Under this definition, the blowing-up property of Theorem~\ref{mainthm}(ii) implies that the Julia set of $S$ is the whole of $\R^d$. For $d\ge3$, this is the first example of a quasiregular map of transcendental type with Julia set equal to $\R^d$.
\item Siebert \cite{S} has shown that every transcendental type quasiregular map has infinitely many periodic points of all periods greater than one.
\item Other quasiregular versions of the sine function have been constructed by Drasin \cite{Dr} and Mayer \cite{Ma}.
\item In \cite{FN3} the authors construct a quasiregular mapping $\R^3 \to \R^3 \cup \{ \infty \}$ with poles, analogous to the tangent function, for which properties like those in Theorem~\ref{mainthm} and Corollary~\ref{cor} hold on a plane in $\R^3$. Further, the family of iterates of this mapping is equicontinuous away from this plane.
\end{enumerate}
\end{remarks}

\section{Quasiregular mappings} \label{sect:qr}

Quasiregular mappings are the natural higher-dimensional generalization of holomorphic functions on the complex plane. A continuous function $f:U\to\R^d$ on a domain $U\subseteq\R^d$ is called \emph{quasiregular} if it belongs to the Sobolev space $W^1_{d,\mathrm{loc}}(U)$ and if there exists $K\ge1$ such that
\[ \|Df(x)\|^d \le KJ_f(x) \quad \mbox{for almost every } x\in U, \]
where $\|Df(x)\|$ is the norm of the derivative of $f$ and $J_f(x)$ denotes the Jacobian determinant. Moreover, this condition implies that there exists $K' \ge1$ such that
\[ J_f(x) \leq K' \ell(Df(x))^d \quad \mbox{for almost every } x\in U,\]
where
\[\ell(A) = \inf_{\|h\|=1} \|Ah\|\]
for a linear map $A$. More details on quasiregular mappings may be found in \cite{R}. We mention only that non-constant quasiregular maps are discrete, open and differentiable almost everywhere. Also, the iterates of a quasiregular map are themselves quasiregular, although the dilatation constants $K$ and $K'$ for the $n$th iterate $f^n$ may become arbitrarily large as $n$ increases.

\section{A quasiregular version of the sine function}\label{sect:qrsine}

\subsection{Construction}

In this section, we briefly recall the construction of the quasiregular version of the sine function given in \cite{BE}.
Fix $d \geq 2$, and write $x=(x_1,\ldots,x_d)$ for an element in $\R^d$ and $\|x\|$ for its Euclidean norm. We denote by $B(x,t)$ the ball of radius $t$ centred at $x$. Let $H_0$ be the hyperplane $\{x \in \R^d : x_d=0\}$.

We start with a bi-Lipschitz mapping $F$ from the half-cube
$[-1,1]^{d-1} \times [0,1]$ to the upper half-ball
\[ \{ x \in \R^d : \|x\| \leq 1, x_d \geq 0 \} \]
which maps the face $[-1,1]^{d-1} \times \{ 1 \}$ onto the hemisphere $\{ x \in \R^d : \|x\| = 1, x_d \geq 0\}$. We additionally assume that $F$ fixes the origin. Such a map $F$ is explicitly constructed in \cite[\S4]{BE}.
We extend $F$ to a mapping $F:[-1,1]^{d-1} \times [0,\infty)\to \{ x \in \R^d : x_d \geq 0 \}$ by
\[ F(x) = e^{x_d-1}F(x_1,\ldots, x_{d-1},1)\]
for $x_d > 1$. Using repeated reflections in hyperplanes, $F$ is extended to a self-mapping of $\R^d$.

The quasiregularity of $F$ is discussed in \cite[\S5]{BE}. In particular, $F$ is differentiable almost everywhere and it is easily shown in \cite[p.167]{BE} that
\[ \beta := \essinf_{x\in\R^d} \ell (DF(x))>0. \]
We now choose $\lambda>4\sqrt{d-1}/\beta$ and consider the function $S:=\lambda F$. This function is locally uniformly expanding in the sense that
\eqn \alpha:=\essinf_{x\in\R^d} \ell (DS(x))=\lambda\beta > 4\sqrt{d-1} >1. \label{defn alpha} \eqnend
See Remark~\ref{correm}. The purpose of this article is to show that this function $S$ has the dynamical properties described in Theorem~\ref{mainthm}.

Following \cite{BE}, we let $\mathcal{S} = \Z^{d-1} \times \{-1,1\}$ and, for $r=(r_1,\ldots,r_d) \in \mathcal{S}$, we define a fundamental half-beam
\[ T(r) = \{ x \in \R^d : |x_j-2r_j|\leq 1, \text{ for } 1 \leq j \leq d-1,\ r_dx_d \geq 0 \},\]
noting that $T_0:=T(0,\ldots,0,1)$ is the initial half-beam in which $F$ was defined. Observe that $S$ maps $T(r)$ bijectively onto either the half-space $\{x_d \geq 0\}$ or $\{x_d \leq 0\}$. We will say that $T(r)$ and $T(s)$ are \emph{adjacent} half-beams if they intersect at a common boundary face. Note that $S$ maps adjacent half-beams onto different half-spaces. The next lemma follows immediately from the construction of $S$.

\begin{lemma}\label{lem:beams}
Let $T(r)$ and $T(s)$ be two adjacent half-beams. Then for any open set $U\subseteq T(r)\cup T(s)$, the  restricted map $S:U\to S(U)$ is a homeomorphism.
\end{lemma}

For $r\in\mathcal{S}$, we denote the corresponding full beam by
\[ T^*(r) = \{ x \in \R^d : |x_j-2r_j|\leq 1, \text{ for } 1 \leq j \leq d-1 \}. \]
Since $T^*(r)$ is the union of the adjacent half-beams $T(r_1,\ldots,r_{d-1},1)$ and $T(r_1,\ldots,r_{d-1},-1)$, Lemma~\ref{lem:beams} implies that $S$ maps the interior of $T^*(r)$ homeomorphically onto its image.

\begin{remark}\label{correm}
The conditions we have placed on $S=\lambda F$ are more rigid than those in \cite{BE} in two ways. Firstly we ask that $F(0)=0$, and secondly we have taken a larger value of $\lambda$ to ensure that $\alpha=\lambda\beta>4\sqrt{d-1}$ in \eqref{defn alpha}, rather than just requiring  that $\alpha>1$ as in \cite[(1.2)]{BE}.

We remark here that Corollary~\ref{cor} holds without these strengthened conditions. Assuming $\alpha>1$, it can be proved that $I(S)^c$ is dense in $\R^d$ by repeated use of Lemma~\ref{lem1} below together with the observation that, for all half-beams $T(r)$, we have $\partial T(r) \cap I(S)=\emptyset$ because
\[ S(\partial T(r))\subseteq H_0 \quad \mbox{ and } \quad S(H_0) \subseteq H_0 \cap B(0,\lambda). \]
\end{remark}

\subsection{Expansion properties of $S$} \label{sect:expansion}

By considering the inverse function of $S|_{T(r)}$, it follows from \eqref{defn alpha} that (as in \cite[(2.3)]{BE})
\eqn  \|S(a)-S(b)\| \ge \alpha\|a-b\| \quad \mbox{for } a,b\in T(r), \ r\in\mathcal{S}. \label{expd ineq} \eqnend
This uniform expansion property of $S$ in half-beams will play an important role in the proof of Theorem~\ref{mainthm}.
We use it now to deduce the next two lemmas.

\begin{lemma}\label{lem1}
If $B(x,t)$ lies in some half-beam $T(r)$, or if $x\in H_0$ and $B(x,t)$ lies in some beam $T^*(r)$, then
\[ B(S(x),\alpha t) \subseteq S(B(x,t)). \]
\end{lemma}
\begin{proof}
Let $w$ be a point on the boundary of $S(B(x,t))$. Since $S$ is an open map, there is $y\in\partial B(x,t)$ such that $S(y)=w$. The points $x$ and $y$ lie in a common half-beam, so \eqref{expd ineq} gives that
\[ \|w-S(x)\| = \|S(y)-S(x)\| \ge \alpha\|y-x\| = \alpha t \]
and the lemma follows.
\end{proof}

\begin{lemma}\label{lem2}
If $x\in H_0$ and $B(x,t)$ lies in some beam $T^*(r)$, then we can find $y\in H_0$ and a beam $T^*(s)$ such that
\[ B(y,\min\{2t,\tfrac{1}{2}\}) \subseteq S(B(x,t))\cap T^*(s). \]
\end{lemma}
\begin{proof}
Write $\delta=\min\{2t,\tfrac{1}{2}\}$ and choose a beam $T^*(s)$ that contains $S(x)$. Since $\delta\le\tfrac12$, we can find a $(d-1)$-dimensional square $Q=[a_1,a_1+2\delta]\times\ldots\times[a_{d-1},a_{d-1}+2\delta]\times\{0\}$ in $H_0\cap T^*(s)$ with one vertex at $S(x)$ and side length $2\delta$. Let $y\in H_0$ be the centre of $Q$. Then $B(y,\delta)\cap H_0 \subseteq Q$ and so $B(y,\delta)\subseteq T^*(s)$. Moreover, if $z\in B(y,\delta)$ then
\begin{align*}
 \| z-S(x)\| &\le \|z-y\| + \|y-S(x)\| \\
 &\le \delta + \delta\sqrt{d-1} \le 2\delta\sqrt{d-1} \le 4t\sqrt{d-1} < \alpha t
\end{align*}
using \eqref{defn alpha}. Thus it follows from Lemma~\ref{lem1} that
\[ B(y,\delta) \subseteq B(S(x),\alpha t) \subseteq S(B(x,t)). \qedhere \]
\end{proof}

\section{Density of periodic points}

In this section we establish Theorem~\ref{mainthm}(i) by proving that the periodic points of $S$ are dense in $\R^d$. The proof of Theorem~\ref{mainthm}(ii) can be based on the same argument and is given in the next section.

It will suffice to show that there is a periodic point of $S$ in every open ball $U_0$ that is compactly contained in the interior of the half-beam $T_0$. To see this, note that for any non-empty open set $U$ in $\R^d$, there is such a ball $U_0$ with $S^2(U_0)\subseteq U$ (by the paragraph preceding Lemma~\ref{lem:beams}). Thus, if we can find a periodic point $p\in U_0$, then $S^2(p)$ will be a periodic point in $U$.

We now fix an open ball $U_0$ satisfying $\overline{U_0}\subseteq\inte T_0$. The strategy of the proof is to find a sequence of open sets $U_j$ such that
\begin{itemize}
\item $U_{j+1} \subseteq S(U_{j})$ for $j \geq 0,$
\item each $U_j$ is contained in the union of two adjacent half-beams,
\item there exists $N \in \N$ such that $\overline{U_0}\subseteq U_N$.
\end{itemize}
If these conditions are met, then by Lemma~\ref{lem:beams} there is a continuous branch of the inverse mapping $S^{-1}:U_{j+1}\to U_j$ for each $j$. Composing these inverse branches yields a continuous branch of $S^{-N}$ mapping $U_N$ into $U_0$. In particular, since $\overline{U_0}\subseteq U_N$, we have a continuous function  $\widetilde{S}:\overline{U_0}\to U_0$. The Brouwer fixed point theorem now states that $\widetilde{S}$ has a fixed point in $U_0$, and so $S$ has a periodic point of period $N$ in $U_0$, which proves Theorem~\ref{mainthm}(i).

It remains to construct the sequence of open sets $U_j$. By using Lemma~\ref{lem1}, we see that there is a maximal integer $n\ge0$ such that $S^n(U_0)$ is contained in a half-beam. For $0<j\le n$, we take $U_j=S^j(U_0)$.

By our choice of $n$, the image $S(U_n)=S^{n+1}(U_0)$ is not contained in any half-beam, so we can find a pair of adjacent half-beams $T(r)$ and $T(s)$ such that $S(U_n)$ intersects $T(r)\cap T(s)$. We may thus choose $U_{n+1}\subseteq S(U_n)$ to be an open ball satisfying
\eqn U_{n+1} \subseteq T(r)\cup T(s)  \quad \mbox{ and } \quad  T(r)\cap T(s)\cap U_{n+1} \ne\emptyset. \label{eqn:U_n+1} \eqnend

Since $U_{n+1}$ intersects the boundary between the half-beams $T(r)$ and $T(s)$, it follows that $S(U_{n+1})$ must intersect the hyperplane $H_0$. As $S(U_{n+1})$ is open, this allows us to find $x\in H_0$ and $\varepsilon>0$ such that the ball $B(x,\varepsilon)$ is contained in both $S(U_{n+1})$ and some full beam $T^*(r')$. We put $U_{n+2}=B(x,\varepsilon)$.

Next we shall use Lemma~\ref{lem2} to choose $U_{n+3},U_{n+4},\ldots$ to be balls of increasing radius, each with centre in $H_0$ and each contained in a beam. In particular, by repeated application of Lemma~\ref{lem2}, we find an integer $m>n+2$ and, for $j=n+3,\ldots,m$, balls $U_j=B(x^j,t_j)$ with the following properties: the centres $x^j\in H_0$; each $U_j$ lies in a beam; $U_{j} \subseteq S(U_{j-1})$; and $t_m=\tfrac12$.

Applying Lemma~\ref{lem1} to $U_m=B(x^m, \tfrac12)$, and recalling \eqref{defn alpha}, shows that
\eqn B(S(x^m),2\sqrt{d-1}) \subseteq S(U_m). \label{eqn:S(U_m)} \eqnend
 By construction, $S(z)=0$ if and only if $z=(2n_1,\ldots,2n_{d-1},0)$, where $n_i\in\Z$. We observe that the ball in \eqref{eqn:S(U_m)} is centred in $H_0$ and has a sufficiently large radius that it must contain a ball $B(z,1)$ such that $S(z)=0$. We take $U_{m+1}=B(z,1)$, noting that $U_{m+1}\subseteq S(U_m)$ by \eqref{eqn:S(U_m)}.

The final stage in the definition of our sequence of sets is to take
\eqn U_j = B(0,\alpha^{j-m-1})\cap\inte T_0, \quad  j\ge m+2. \label{eqn:j>m+2} \eqnend
We show that $U_{j+1} \subseteq S(U_{j})$ for $j \ge m+2$; the proof that $U_{m+2} \subseteq S(U_{m+1})$ is similar. If $y\in U_{j+1}$ for some $j \ge m+2$, then we have that $y\in\inte T_0\subseteq S(\inte T_0)$, and so there exists $w\in\inte T_0$ such that $S(w)=y$. Since $S(0)=0$, it follows from \eqref{expd ineq} and \eqref{eqn:j>m+2} that
\[ \alpha\|w\| \le \|S(w)\| = \|y\| < \alpha^{j-m}. \]
We deduce that $w\in U_{j}$ and therefore $y\in S(U_{j})$ as required.

Finally, by recalling that $\alpha>1$ and that $\overline{U_0}\subseteq\inte T_0$, we see from \eqref{eqn:j>m+2} that there is a large $N$ such that the ball $\overline{U_0}\subseteq U_N$. Therefore, the three conditions stated earlier are satisfied and the proof of Theorem~\ref{mainthm}(i) is complete.

\section{Blowing-up property of $S$}

We will prove Theorem \ref{mainthm}(ii) by showing that given any non-empty open set $U \subseteq \R^d$ and any $R>0$, there exists $k \in \N$ such that $B(0,R) \subseteq S^k(U)$.

We will apply the argument of the preceding section and so, given a non-empty open $U\subseteq \R^d$, our first step is to take a ball $U_0\subseteq T_0$ such that $S^2(U_0)\subseteq U$. Following the proof given earlier, we find $m\in\N$ such that $B(z,1)\subseteq S^{m+1}(U_0)$, where $S(z)=0$. This implies that $S^m(U)$ contains a neighbourhood of the origin (and, in fact, $S^m(U)$ contains $B(0,\alpha)$).

We now claim that, for any $t>0$,
\eqn S(B(0,t)) \supseteq B(0,\alpha t). \label{Boat} \eqnend
Since $\alpha>1$, repeated use of \eqref{Boat} will now be enough to establish Theorem~\ref{mainthm}(ii). The proof of the claim is similar to the final part of the preceding section. We recall that $S$ maps the beam $T^*(0)$ containing the origin onto $\R^d$. Hence, if $y\in B(0,\alpha t)$, then there is $w\in T^*(0)$ such that $S(w)=y$. Since $S(0)=0$, \eqref{expd ineq} gives that
\[ \alpha\|w\| \le \|S(w)\| = \|y\| < \alpha t, \]
from which it follows that $w\in B(0,t)$ and therefore \eqref{Boat} holds.

\end{document}